\documentclass[a4paper, 11pt]{amsart}

\usepackage{amsmath,amssymb,amscd,enumerate,setspace}
\usepackage{amsthm}
\usepackage{latexsym}

\usepackage{xcolor}

\usepackage{mdframed}

\newmdenv[backgroundcolor=yellow]{shaded}
\usepackage{hyperref}
\input xypic
\xyoption{all}

\newcommand{\rar}{\rightarrow}
\newcommand{\lar}{\longrightarrow}

\newcommand{\llar}{-\kern-5pt-\kern-5pt\longrightarrow}
\newcommand{\lllar}{-\kern-5pt-\kern-5pt\llar}

\newtheorem{Theorem}{Theorem}[section]

\newtheorem{Corollary}[Theorem]{Corollary}
\newtheorem{Proposition}[Theorem]{Proposition}

\newtheorem{Conjecture}[Theorem]{Conjecture}

\newtheorem{Remark}[Theorem]{Remark}
\newtheorem{Example}[Theorem]{Example}
\newtheorem{Definition}[Theorem]{Definition}
\newtheorem{Question}[Theorem]{Question}
\newtheorem{Questions}[Theorem]{Questions}

\newtheorem{Sticky Points}[Theorem]{Sticky Points}

\def\codim{\mbox{\rm codim}}
\def\coker{\mbox{\rm coker}}

\def\ds{\displaystyle}

\def\Ext{\mbox{\rm Ext}}

\def\gr{\mbox{\rm gr}}

\def\grade{\mbox{\rm grade}}
\def\height{\mbox{\rm ht}}
\def\Hom{\mbox{\rm Hom}}
\def\ker{\mbox{\rm ker}}

\def\e{\mathrm{e}}

\def\rme{\mathrm{e}}

\def\m{\mathfrak{m}}

\def\ZZ{\mathbb{Z}}

\def\AA{{\mathbf A}}

\def\BB{{\mathbf B}}
\def\CC{{\mathbf C}}

\def\RR{{\mathbf R}}
\def\SS{{\mathbf S}}

\def\LL{{\mathbf L}}

\def\TT{{\mathbf T}}
\def\ff{{\mathbf f}}
\def\pp{{\mathbf p}}

\def\g2{{\mathbf g}}

\def\aa{{\mathbf a}}

\def\p{{\mathfrak p}}
\def\q{{\mathfrak q}}

\def\m{{\mathfrak m}}
\def\M{\mathfrak{M}}

\def\C{\mathcal{C}}
\def\B{\mathcal{B}}
\def\D{\mathcal{D}}

\def\ddeg{\mbox{\rm bideg}}
\def\tdeg{\mbox{\rm tdeg}}

\def\cdeg{\mbox{\rm cdeg}}
\def\codim{\mbox{\rm codim}}

\def\coker{\mbox{\rm coker}}

\def\tr{\mbox{\rm trace}}

\begin{document}

\title{\sc  The Bi-Canonical Degree of a Cohen-Macaulay Ring }

\author{L. Ghezzi} \address{Department of Mathematics, New York City
College of Technology-Cuny, 300 Jay Street, Brooklyn, NY 11201, U.S.A.} \email{lghezzi@citytech.cuny.edu}
\author{S. Goto}
\address{Department of Mathematics, School of Science and Technology,
Meiji University, 1-1-1 Higashi-mita, Tama-ku, Kawasaki 214-8571,
Japan} \email{goto@math.meiji.ac.jp}
\author{J. Hong}
\address{Department of Mathematics, Southern Connecticut State
University, 501 Crescent Street, New Haven, CT 06515-1533, U.S.A.}
\email{hongj2@southernct.edu}
\author{H. L. Hutson} \address{40 U Madison Park Gardens, Port Washington, NY 11050, U.S.A.}
\email{hutsonroy@gmail.com}
	\author{W. V. Vasconcelos}
\address{Department of Mathematics, Rutgers University, 110
Frelinghuysen Rd, Piscataway, NJ 08854-8019, U.S.A.}
\email{wvasconce@gmail.com}

\thanks{The third author is partially supported by the SCSU Faculty Creative Activity Research Grant Program (AY 2018-2019).}

\maketitle

{\em \small Dedicated to Professor Craig Huneke  on the occasion of his birthday  for his groundbreaking contributions to Algebra,  particularly  to Commutative Algebra.}

\begin{abstract}
This paper is a sequel to \cite{blue1} where we introduced an invariant, called canonical degree, of Cohen-Macaulay local rings
 that admit a canonical ideal. Here to each such ring $\RR$ with a canonical ideal, we attach a different invariant, called bi-canonical degree, which in dimension $1$
appears also in \cite{Herzog16} as the residue of $\RR$. The minimal values of these functions characterize specific classes of
Cohen-Macaulay rings.  We give a uniform presentation of such degrees and discuss some computational opportunities offered by the
  bi-canonical degree.
  \end{abstract}

\section{Introduction}

\noindent  Let $(\RR, \m)$ be a Cohen-Macaulay local ring of dimension $d$ that has a canonical ideal $\C$. 
We will use \cite{BH} as our reference for the basic properties of these rings and their terminology.
Our central viewpoint is to look at the properties of $\C$ as a way to refine our understanding of $\RR$.  Recall that $\RR$  is Gorenstein when $\C$ is isomorphic to $\RR$.
In \cite{blue1}  we treated metrics  aimed at measuring the deviation from $\RR$ being Gorenstein.
Here we explore another pathway but still with the same overall goal. Unlike \cite{blue1} the approach here is arguably more suited for   computation in classes of algebras such as  monomial subrings (see Proposition~\ref{preaGor}), Rees algebras (see  Section \ref{Reesalgs}), and numerous classes of rings of dimension one (see Theorem~\ref{TCdeg}).
 First however we outline the general
 underpinning of these developments.
 The organizing principle in  setting up a canonical degree
is to recast numerically   criteria for a Cohen-Macaulay ring  to be Gorenstein. 
Most of our results are in a ring of dimension one but we also treat  higher dimensional rings in some special cases.

\medskip

We shall now describe how this paper is organized. For a Cohen-Macaulay local ring $(\RR, \m)$ of dimension $d$ with a canonical ideal $\C$,
we  attach a non-negative integer $c(\RR)$ whose value reflects divisorial properties of $\C$
and provides for a stratification of the class of Cohen-Macaulay rings.
 We have noted two such functions, the residue of $\RR$ and the canonical degree,
  in the current
literature (\cite{Herzog16}, \cite{blue1})  and here we build a third degree, the bi-canonical degree. 

\medskip

In Section 2 we recall from the literature the needed facts to put together  the degrees. In Section 3, we define the new degree called {\em bi-canonical} degree of $\RR$. It 
is given by
\[ \ddeg(\RR)= \deg(\C^{**}/\C) = \sum_{\tiny \height(\p)=1} \ddeg(\RR_{\p}) \deg(\RR/\p) = \sum_{\tiny \height(\p)=1} [\lambda(\RR_{\p}/\C_{\p}) - \lambda(\RR_{\p}/\C^{**}_{\p})] \deg(\RR/\p), \] where $\C^{**}$ is the bidual of 	$\C$. This is a well-defined finite sum independent of the chosen canonical ideal $\C$.
These constructions are useful for effective symbolic calculation (we used Macaulay2 (\cite{Macaulay2}) in our
experiments) but turn out useful for theoretical calculations in special classes of rings. 
It leads immediately to comparisons to two other degrees, the {\em canonical} degree of \cite{blue1},
$\cdeg(\RR)= \deg(\C/(s))$ for a minimal reduction $(s)$ of $\C$ in dimension one and suitably defined as above to all dimensions,
 and the {\em residue} of $\RR$ of \cite{Herzog16}
$\tdeg(\RR) = \deg(\RR/\mbox{\rm trace}(\C))$, where $\mbox{\rm trace}(\C)$ is the trace ideal of $\C$.
Arising naturally is a comparison conjecture, that is $\cdeg(\RR) \geq \ddeg(\RR)$. We engage in a discussion on how to recognize that a codimension-one ideal $I$ is actually a canonical ideal. We finally recall the notion of the {\em rootset}
of $\RR$ (\cite{blue1}), perhaps one of least understood sets attached to $\C$ and raise questions on how it affects the values of the degrees.

 \medskip

We start in Section 4 a study of algebras according to the values
 of $c(\RR)$, that is of either  $\cdeg(\RR)$ or $\ddeg(\RR)$.  If $c(\RR)=0$, for all canonical degrees, then $\RR$ is Gorenstein in
codimension one. It is natural to ask which rings correspond to the small values of $c(\RR)$. 
In dimension one, $\cdeg(\RR)\geq r(\RR)-1$ and $\ddeg(\RR)\geq 1$,
where equality corresponds to the {\em almost Gorenstein } rings of  (\cite{BF97, GMP11, GTT15})
and
{\em nearly Gorenstein} rings of \cite{Herzog16}.

\medskip

We begin calculations of $\cdeg(\RR)$ and $\ddeg(\RR)$  in Sections 5, 6, 7 and 8 for various classes of algebras. Unlike the case of $\cdeg(\RR)$, already for hyperplane sections the
behavior of $\ddeg(\RR)$ is more challenging. Interestingly, for monomial rings $k[t^a, t^b, t^c]$ the technical difficulties are reversed. In two cases, augmented rings and
[tensor] products, very explicit formulas are derived.  More challenging is the case of Rees algebras when we are often
 limited to deciding the vanishing of degrees. The most comprehensive results resolve around $\m:\m$.

\section{Setting up and calculating canonical degrees}

We make use of the  basic facts expressed in the codimension $1$ localizations of $\RR$.

\begin{Theorem} \label{corecan}
Let $(\RR, \m)$ be a local ring of dimension $1$
and let $Q$ be  its
total ring of fractions. Assume that $\RR$ has a canonical module $\omega_{\RR}$.

\begin{enumerate}[{\rm (1)}]
\item $\RR$ has a canonical ideal if and only if the total ring of fractions of  $\widehat{\RR}$ is Gorenstein. \cite{Aoyama, BrodSharp, HK2}.
\item An $\m$-primary ideal $I$ is a canonical ideal if and only if $I:_{\RR} \m = I:_Q \m = (I, s)$  for some $s \in \RR$. \cite[Theorem 3.3]{HK2}. 
\item $\RR$ is Gorenstein if and only if $\omega_{\RR}$ is a reflexive module. \cite[Corollary 7.29]{HK2}.
\item If $\RR$ is an integral domain with finite integral  closure,  then $I^{**}=\Hom( \Hom(I, \RR), \RR)$ is integral over $I$. \cite[Proposition 2.14]{CHKV}.
\end{enumerate}
\end{Theorem}

We often assume harmlessly that $\RR$ has  an infinite residue field.  For a finitely generated $\RR$-module $M$,
 the notation $\deg(M)=\rme_0(\m, M)$ refers to the multiplicity  defined by the $\m$-adic topology.
For an $\m$-primary ideal $I$, $\rme_0(I)$ is the multiplicity of $\RR$ defined by the $I $-adic topology.
The Cohen-Macaulay type of $\RR$ is denoted by $r(\RR)$.  In general, in writing $\Hom_{\RR}$ we omit the symbol for the
 underlying ring.

 \begin{Remark}{\rm 
 Let $(R, \m)$ be a Cohen-Macaulay local ring with a canonical ideal $\C$.  
 \begin{enumerate}[(1)]

 \item The canonical degree $\cdeg(\RR)$ of $\RR$ was introduced in \cite{blue1}.  In dimension $1$,  we select an  appropriate
 regular element $c$ of $\C$ and define $\cdeg(\RR) = \deg(\C/(c))$.  We require choices that yield the same value.
 In \cite{blue1}, the ideal $(c)$ is picked as a minimal reduction of $\C$ so that  $\cdeg(\RR) = \rme_0(\C) - \deg(\RR/\C)$. 
 \item The bi-canonical degree $\ddeg(\RR)$ of $\RR$ is defined in Theorem~\ref{genddeg1}. 
  Let $\C^{*} = \Hom(\C, \RR)$ be the dual of $\C$. We define $\ddeg(\RR) = \deg(\C^{**}/\C)$, where $\C^{**}$ is the bidual of $\C$.
 \item  The trace degree $\tdeg(\RR)$ of $\RR$ was introduced in \cite{Herzog16} and called the {\em residue} of $\RR$.
  A standard metric is   $\tdeg(\RR) =\deg(\RR/\tr(\C) )$, where  $\tr(\C) = \C\cdot \C^{*} $ is the {\em trace} ideal of $\C$. Recall that
  the trace of an ideal $I$ is 
   $\tr(I)= I\cdot I^{*} =
   \{f(x): f\in I^{*}, x\in I \}$. 
 \item The  numbers $\cdeg( \RR)$, $\ddeg(\RR) $ and $\tdeg(\RR)$ are independent of the choice of $\C$. 
 In dimension $1$,  $\ddeg(\RR)$ equals $\tdeg(\RR)$ (see Proposition~\ref{resisddeg1}).
 They may differ in dimension $d>1$.
\item  In dimension $1$, $\RR$ is Gorenstein if and only if one of $\cdeg(\RR)$, $\ddeg(\RR)$ or $\tdeg(\RR)$ vanishes, in which case all three vanish.
\item  If $\RR$ is not Gorenstein, $\cdeg(\RR) \geq r(\RR) - 1\geq 1 $,  $\tdeg(\RR) \geq 1$, $\ddeg(\RR) \geq 1$.  The cases when the minimal values are reached have the following designations.

 \begin{enumerate}[(i)]
 \item  $\cdeg(\RR) = r(\RR) - 1$ if and only if $\RR$ is  an {\em almost Gorenstein ring}.
  \item $\tdeg(\RR) = 1$ if and only if $\RR$ is a {\em nearly Gorenstein ring}.
  \item  $\ddeg(\RR) = 1$ if and only if  $\RR$ is  a {\em Goto ring}.
 \end{enumerate}
\end{enumerate}
}\end{Remark}

 There are some  relationships among these invariants.

\begin{Proposition}{\rm [J. Herzog, personal communication]}\label{resisddeg1}  Let $\RR$ be a Cohen-Macaulay local ring of dimension $1$ with a canonical ideal $\C$. Then
\[ \ddeg(\RR) =
\lambda(\RR/\tr(\C)) = \tdeg(\RR).\]
\end{Proposition}

\begin{proof} 
From $\tr(\C) = \C\cdot
 \C^{*}$, we
 have
\[  \Hom(\tr(\C),\C)=  \Hom(\C \cdot \C^{*},\C)= \Hom(\C^{*}, \Hom(\C,\C))  = \C^{**}.\]
 Now dualize the
  exact sequence
\[
0 \rar \tr(\C) \rar  \RR \rar \RR/\tr(\C)\rar 0,\]   into $\C$
to obtain the following:
\[0 \rar  \C \rar \C^{**} \rar  \Ext^1(\RR/\tr(\C),\C)\rar  0,\]
which
 shows that $\C^{**}/\C$ and $\Ext^1(\RR/\tr(\C),\C)$ are isomorphic. Since by local duality $\Ext^1(\cdot,\C)$ is self-dualizing on modules of
 finite length,
$\ddeg(\RR)= \lambda(\RR/\tr(\C)).$ 
\end{proof}

\section{The bi-canonical degree}

\noindent
The approach in \cite{blue1} is dependent on finding minimal reductions, which is a particularly hard task.
We pick here one that seems more amenable to computation.
In the natural embedding
\[ 0 \rar \C \rar \C^{**} \rar \B \rar 0 \]
the module $\B$ remains unchanged when $\C$ is replaced by another canonical module.

 \begin{Theorem}\label{genddeg1} Let $(\RR, \m)$ be a Cohen-Macaulay local ring  of dimension $d \geq 1$  that has a canonical ideal $\C$.
  Then
  \[ \ddeg(\RR)= \deg(\C^{**}/\C) =
     \sum_{\tiny \height \p=1} \ddeg(\RR_{\p}) \deg(\RR/\p) = \sum_{\tiny \height \p=1} [\lambda(\RR_{\p}/\C_{\p}) - \lambda(\RR_{\p}/\C^{**}_{\p})] \deg(\RR/\p)
  \] is a well-defined finite sum independent of the chosen canonical ideal $\C$.
  Furthermore,
  $\ddeg(\RR)\geq 0$ and vanishes if and only if  $\RR$ is Gorenstein in codimension $1$.
\end{Theorem}

\begin{proof}
If $d=1$, then $\B$ vanishes if and only if $\RR$ is Gorenstein (see Theorem \ref{corecan}-(3)). 
Suppose $d >1$. 
Then  $\B$ embeds into the Cohen-Macaulay module $\RR/\C$ that  has dimension $d-1$, and thus $\B$ either is zero or its associated primes are associated primes of $\C$, all of which have codimension $1$.
\end{proof}

 As in \cite{blue1}, we  explore the length of $\B$, $\ddeg(\RR) =\lambda(\B)$ which we view as a degree
 in dimension $1$, and $\deg(\B)$ in general. We focus on  $1$-dimensional case. 
In this case, if $(c)$ is a minimal reduction of $\C$, then
\[ \cdeg(\RR) = \lambda (\RR/(c)) - \lambda(\RR/\C), \; \mbox{and} \;  \ddeg(\RR) = \lambda(\RR/\C) -\lambda(\RR/\C^{**}).\]

\begin{Conjecture}{\rm (Comparison Conjecture) \label{cdegvsfddeg}
In general,   $\cdeg(\RR) \geq \ddeg(\RR)$. 
}\end{Conjecture}

The point to be raised is which of  $\e_0(\C)$ or $\lambda(\RR/\C^{**})$  is more approachable.   We argue that the latter is more efficient according to the method
of computation below.

\subsubsection*{Computation of duals  and biduals} Throughout this section, let $Q$ be the total ring of fractions of $\RR$. If $I$ is a regular ideal of the Noetherian ring $\RR$, then ${\ds  \Hom(I, \RR) = ( \RR:_Q I )}$.  A difficulty is that Computer Systems such as Macaulay2  (\cite{Macaulay2}) are set to calculate quotients of the form $(A:_\RR B)$ for two ideals $A,B\subset \RR$, which is done with calculations of syzygies.  But it is possible to compute quotients in the total rings of fractions as follows.

\begin{Proposition} \label{bidual}
Let $I$ be an $\RR$-ideal and suppose  $a$ is a regular element of $I$.

 \begin{enumerate}[{\rm (1)}]
  \item $\Hom(I,J)=a^{-1}(aJ:_{R} I)$. In particular,  $I^{*} = \Hom(I,\RR) = a^{-1} ((a) :_\RR I)$.
 \item  $I^{**} = \Hom(\Hom(I, \RR), \RR) =( (a):_\RR ((a):_\RR I))$, which is the annihilator of $\Ext^1_{\RR}(\RR/I, \RR)$.
   \item $\tr(I) = I \cdot I^{*} = a^{-1} I  \cdot ((a):_\RR I)$.
\end{enumerate}
     \end{Proposition}
\begin{proof}  (1) It follows from
\[ q \in \Hom(I, J) = (J :_Q I) \; \Longleftrightarrow \; 
 qI \subset J   \; \Longleftrightarrow  \;    aqI \subset aJ  \; \Longleftrightarrow   \;  q \in a^{-1} (aJ:_\RR I). 
\]
(2) It follows from 
\[ I^{**} = [a^{-1} ((a):_\RR I)]^{*}= a ( \RR:_Q ((a):_\RR I)) = a ( a^{-1} (
(a):_\RR ((a):_\RR I) )= (a):_\RR ((a):_\RR I). \]
See also \cite[Remark 3.3]{Vas91}, \cite[Proposition 3.1]{Herzog16}.
\end{proof}

\subsubsection*{Recognition of canonical ideals} An interesting question is to determine whether a given ideal is a canonical ideal. 
 The following  observations are influenced by the discussion in \cite[Section 2]{Elias}. For other methods, see \cite{GMP11}.

\begin{Proposition}  \label{recog}
Let $(\RR,\m)$ be a Cohen-Macaulay local ring of dimension $1$. Then an
$\m$-primary ideal $I$ is a canonical ideal if   $I$ is  an irreducible ideal  and $\Hom(I,I) = \RR$.
 \end{Proposition}

\begin{proof}  Since $I$ is irreducible $(I:_{\RR} \m) = (I,s)$ for some $s \in \RR$. 
If $q \in I:_Q \m$, then $q I \subset q \m \subset I$. The means that  $q\in \Hom(I, I)= \RR$. Therefore, $(I:_{Q} \m)=(I:_{\RR} \m)$. By Theorem~\ref{corecan}-(2), the ideal $I$ is a canonical ideal. 
  \end{proof}

 \begin{Proposition} \label{recogcor} Let $(\RR, \m)$ be a Cohen-Macaulay local ring of dimension $1$.
Then $\m$ is a canonical ideal if and  only if $\RR$ is a discrete valuation ring {\rm(DVR)}.
\end{Proposition}

\begin{proof}
\noindent Suppose that $\m$ is a canonical ideal. Let $x \in \m$ be a regular element. Then there exists a positive integer $n$ such that $\m^n \subset x \RR$ but $\m^{n-1} \not \subset x\RR$.  If $n = 1$,  then $\m = x\RR$ and there is nothing else to prove. Suppose that $n \geq 2$. Then $\m^{n} x^{-1}$ is an $\RR$--ideal. If $\m^{n} x^{-1} \subset \m$, then $\m^{n-1} x^{-1} \subset ( \m :_{Q} \m) = \Hom(\m, \m) = \RR$. This means that $\m^{n-1} \subset x \RR$, which is a contradiction. Thus, we have $\m (\m^{n-1}x^{-1}) = \m^{n} x^{-1} = \RR$. Hence $\m$ is invertible, or $\RR$ is a discrete valuation ring. The converse follows from Proposition~\ref{recog}.
\end{proof}

\subsubsection*{The roots of $\RR$} There are several properties of $\C$, such as its {\em reduction number}, that
impact the canonical degree of the ring (see \cite{blue1}). Another property is the
 {\em rootset} of $\RR$ consisting  of the ideals $L$ such that $L^n \simeq \C$ for some integer $n$.
They are distributed into finitely many isomorphism classes. The rootset was studied in \cite{blue1} for  its role in examining properties
of canonical ideals. Here
is one instance.

\begin{Proposition}
Let $\RR$ be a Cohen-Macaulay local ring of dimension $1$ with a canonical ideal $\C$. Let $L$ be an ideal such that $L^n\simeq \C$ for some integer $n$. 
If $L$ is an irreducible ideal,  then $\RR$ is a Gorenstein ring.
\end{Proposition}

\begin{proof} Since $L^n\simeq \C$ for some integer $n$, we obtain the following. 
\[ \Hom(L, L) = (L :_{Q} L) \subset (L^{n} :_{Q} L^{n}) = (\C :_{Q} \C )  = \Hom(\C, \C) = \RR. \]
Since $L$ is irreducible and $\Hom(L, L)=\RR$, by Proposition~\ref{recog}, we get $L \simeq \C$. Moreover, we have $\C \simeq \C^{n}$. Hence
\[ \C \simeq \C^{n} = \C^{n-1} \cdot \C \simeq \C^{n-1} \cdot \C^{n} = \C^{2n-1}. \]
By iteration we get $\C \simeq \C^m$ for infinitely many values of $m$. This means that 
\[ \Hom(L^m, L^m) \simeq \Hom(\C^{m}, \C^{m}) \simeq \Hom(\C, \C) = \RR \]  for infinitely many values of $m$. By \cite[Proposition 5.3]{blue1}, $\RR$ is Gorenstein. 
\end{proof}

We conclude this section with the following open questions. 

\begin{Question}{\rm
\begin{enumerate}[(1)]
\item If $L^n \simeq \C$ for some integer $n$ and $L$ is reflexive, then must  $\RR$ be Gorenstein?
\item Suppose there is an ideal $L$ such that $L^2 = \C$. Is $\cdeg(\RR)  $ even?
\item Let $(\RR, \m)$ be a Cohen-Macaulay local ring of dimension $d\geq 2$. Let $\C$ be a Cohen-Macaulay ideal of codimension $1$. What is  a criterion for $\C$ to be a canonical ideal ?
\end{enumerate}
}\end{Question}

\section{Minimal values of canonical degrees}
\noindent
Now we begin to examine the significance of the minimal value of $\ddeg(\RR)$.  First we recall the definition of almost Gorenstein rings
 (\cite{BF97, GMP11, GTT15}).

\begin{Definition}{\rm (\cite[Definition 3.3]{GTT15})  A Cohen-Macaulay local ring $\RR$ with a canonical module $\omega_{\RR}$  is said to be an {\em
almost Gorenstein} ring if there exists an exact sequence of $\RR$-modules
${\ds 0 \rightarrow \RR \rightarrow \omega_{\RR} \rightarrow X \rightarrow  0}$
such that $\nu(X)=\e_0(X)$, where $\nu(X)$ is the minimal number of generators of $X$. In particular if $\dim(\RR)=1$, then $X \simeq (\RR/\m)^{r-1}$, where $r$ is the type of $\RR$.
}\end{Definition}

\begin{Theorem}\label{AGorddeg} Let $(\RR,\m)$ be a Cohen-Macaulay  local ring of dimension $1$
 with a canonical ideal $\C$.  If $\RR$ is almost Gorenstein,  then $\ddeg(\RR) =1$.
\end{Theorem}

\begin{proof} Let $r$ be the type of $\RR$ and $(c)$ a minimal reduction of $\C$. Since $\RR$ is almost Gorenstein, there exists an exact sequence of $\RR$-modules
\[ 0 \rightarrow (c) \rightarrow \C \rightarrow X \rightarrow  0 \]
such that $X \simeq (\RR/\m)^{r-1}$.  By applying $\Hom(\cdot, (c))$ to the exact sequence above, we obtain 
\[ 0\rar  \Hom(\C, (c)) \rar \RR \rar  \Ext^{1}(X, (c)) \rar \Ext^{1}(\C, (c)) \rar 0.\]
The image of $\Hom(\C, (c))$ in $\RR$ is a proper ideal of $\RR$  because  otherwise $\C= (c)$.  Thus, $\Hom(\C, (c)) = \m$. 
On the other  hand, by Proposition~\ref{bidual}, $\C^{**} = \Hom(\m, (c))$ or $\C^{**}/(c)$ is isomorphic to the socle of $\RR/(c)$  in $\RR$.
Therefore, we obtain 
\[\ddeg(\RR) = \lambda(C^{**}/(c)) - \lambda(C/(c)) = \dim(\mbox{Soc}(\RR/(c)) ) - \dim(X) = r-(r-1)= 1. \qedhere\] 
\end{proof}

The example below shows that the
converse does not hold true.

\begin{Example}\label{monoex}{\rm Consider the monomial ring 
  $\RR = \mathbb{Q}[t^5, t^7, t^9]$. 
 Then we have a presentation $\RR  \simeq \mathbb{Q}[X, Y, Z]/P $, with $P = (Y^2-XZ, X^5-YZ^2, Z^3-X^4 Y)$. Let
  us examine some properties of $\RR$. We denote the images of $X, Y, Z$ in $\RR$   by $x, y, z$ respectively.
  The ideal $\C=(x, y)$ is the canonical ideal of $\RR$ (see the statement above Proposition~\ref{preaGor}). Since $(x)$ is a minimal reduction of $\C$, we get  $\cdeg(\RR) = \lambda(\C/(x))  = 2$. On the other hand, the type of $\RR$ is $2$. Thus, $\RR$ is not almost Gorenstein (\cite[Proposition 3.3]{blue1}). However, $\C^{**} = (x):  ((x):\C)$ satisfies $\lambda(\C^{**}/\C) = 1$  by {\em Macaulay2} calculation.
This shows that $\ddeg(\RR) = 1$. This example shows that $\ddeg(\RR) = 1$ holds  in a larger class of  rings than almost Gorenstein rings (for dimension $1$).
}\end{Example}

\begin{Definition}\label{Gotoring} {\rm  A Cohen-Macaulay local ring $\RR$ of dimension $d$ is a {\em Goto ring} if it has a canonical
ideal and $\ddeg(\RR) =1$.
}\end{Definition}

\begin{Proposition}\label{AGorddeg2} Let $(\RR,\m)$ be a Goto ring of dimension $d$ with a canonical ideal $\C$.  Let $\p$ be a prime ideal of height $1$ such that $\RR_{\p}$ is not Gorenstein. 
Suppose that  $\C$ is equimultiple.
\begin{enumerate}[{\rm (1)}]  
\item If $d=2$, then $\C^{**}/\C \simeq \RR/\p$.
\item In general, suppose that $\RR$ is complete and contains a field. Then $\C^{**}/\C \simeq \RR/\p$.
\end{enumerate}
\end{Proposition}

\begin{proof} Since $\ddeg(\RR)=1$, this means $\C^{**}/\C$ is $\p$-primary. 

\smallskip

\noindent (1) Since $d=2$, the module $\C^{**}/\C$ is Cohen-Macaulay of dimension $1$ and multiplicity $1$. Moreover,  $\C^{**}/\C$ is 
 an $\RR/\p$-module of rank $1$ and $\RR/\p$ is a discrete valuation domain. Hence  $\C^{**}/\C \simeq \RR/\p$.
 
\smallskip

\noindent (2)  In all dimensions, $\C^{**}/\C$ has only $\p$ for associated prime and is a torsion-free $\RR/\p$-module of rank $1$. Thus it is
isomorphic to an ideal of $\RR/\p$. Moreover, 
$\C^{**}/\C$ has  the condition $S_2$ of Serre from the exact sequence
\[ 0 \rar \C^{**}/\C \rar \RR/\C \rar \RR/\C^{**} \rar 0,\]
where $\RR/\C^{**}$ has the condition $S_1$. By \cite{Nagata}, $\RR/\p$ is a regular local ring and therefore $\C^{**}/\C \simeq \RR/\p$.
\end{proof}

\section{Change of rings}

\noindent
Let $\varphi:\RR \rar \SS$ be a homomorphism of Cohen-Macaulay rings. We examine
a few cases of the relationship between $\ddeg(\RR)$ and $\ddeg(\SS)$ induced by $\varphi$. We skip polynomial, power series and
completion since it is similar to the treatment of $\cdeg(\RR)$  in \cite[Section 6]{blue1}.  
If $\RR \rar \SS$ is a finite injective homomorphism of Cohen-Macaulay rings
and $\C$ is a canonical ideal of $\RR$, then $\D = \Hom(\SS, \C)$ is a canonical module for $\SS$, according to 
\cite[Theorem 3.3.7] {BH}.
 Recall how $\SS$ acts on
$\D$: If $f\in \D$ and $a, b\in \SS$, then $a\cdot f(b) = f(ab)$.

\subsubsection*{Augmented rings}
A case in point is that of the augmented extensions.
Let $(\RR,\m)$ be an $1$-dimensional Cohen-Macaulay local ring with a canonical ideal $\C$.
Assume that $\RR$ is not a valuation domain.
Let $\AA$ be the augmented ring $\RR \ltimes \m$.
That is, $\AA = \RR\oplus \m \epsilon$, $\epsilon^2 = 0$. (Just to keep the components apart in computations we use $\epsilon$ as a place holder.)
Let $(c)$ be a minimal reduction of the canonical ideal $\C$. We may assume that $\C\subset \m^2$ by replacing $\C$ by $c\C$ if necessary.
Then a canonical module of $\AA$ is  $ \Hom_{\RR} (\AA, \C)$ and, by \cite[Proposition 6.7]{blue1}, we have an $\AA$-isomorphism
\[ \Hom_{\RR} (\AA, \C) \simeq (\C:_{\RR} \m) \times \C. \]
Moreover, $ (\C:_{\RR} \m) \times \C$ is a canonical ideal of $\AA$ because $\C \subset \m^{2}$ and $(\C :_{\RR} \m) \subset \m$.

\begin{Proposition} Let $(\RR, \m)$ be a $1$-dimensional Cohen-Macaulay local ring with a canonical ideal and let $\AA = \RR\ltimes \m$ the augmented ring. Suppose that $\RR$ is not Gorenstein. Then
\[ \ddeg(\AA) = 2\, \ddeg (\RR) -1.\]
In particular if $\RR$ is a Goto ring, then $\AA$ is also a Goto ring.
\end{Proposition}

\begin{proof}
Let $\C$ be the canonical ideal of $\RR$ and $L= \C :_{\RR} \m$. Then 
$\mathcal{D} = L \times \C$ the canonical ideal of $\AA$. Let $Q$ be the total quotient ring of $\RR$. Then 
the total  quotient ring of $\AA$ is $Q \ltimes Q$.  If
$(a,b) \in Q\ltimes Q$ is in
\[ \Hom(\mathcal{D}, \AA)= \Hom(( L \oplus  \C \epsilon ), (\RR\oplus  \m\epsilon)) ,
\] then
$a \C \subset \RR$, or $a\in \C^{*}$. On the other hand, $b  \C\subset \m$ or $b\in \C^{*} $. For the converse, Let $(a, b ) \in  \C^{*} \oplus \C^{*}$.  
Since $\RR$ is not Gorenstein, $a\C \neq \RR$ and 
 $a L \subset \m $. Thus
 \[ \mathcal{D}^{*} = \C^{*} \oplus \C^{*} \epsilon.\]
 Similarly, we have 
\[ \D^{**} = \Hom((\C^{*} \oplus \C^{*}\epsilon ), (\RR\oplus \m \epsilon )) =  \C^{**} \oplus \C^{**} \epsilon .\]
This gives
\[ \D^{**}/\D = ( \C^{**}/\C) \oplus   (\C^{**}\epsilon /  L \epsilon) .\]
Therefore we have
\[ \ddeg(\AA) = \lambda ( \D^{**}/\D ) = \lambda ( \C^{**}/\C) +   \lambda(\C^{**}\epsilon /  L \epsilon) = 2\, \ddeg (\RR) -1. \qedhere \]
\end{proof}

\subsubsection*{Products} Let $k$ be a field, and let $\AA_1, \AA_2$ be two finitely generated Cohen-Macaulay $k$-algebras. Let us look at the canonical degrees of the
{\em product} $\AA = \AA_1 \otimes_k \AA_2$.  If $\AA_i$, $i=1,2$, are localizations of finitely generated $k$-algebras, then we may assume that $\AA$ is an appropriate localization. 

\begin{Proposition}\label{prod} Let $\AA_i$, $i=1,2$, be as above. Suppose that there exist equimultiple canonical ideals $\C_{1}$ and $\C_{2}$ of $\AA_{1}$ and $\AA_{2}$ respectively.
\begin{enumerate}[{\rm (1)}]
\item The product $\AA  = \AA_1 \otimes_k \AA_2$ is Cohen-Macaulay.
\item The canonical ideal of $\AA$ is  $\C = \C_1 \otimes \C_2$. Moreover, we have  $\C^{**} = \C_1^{**} \otimes \C_2^{**}$.
\item We have the following.
\[ \cdeg(\AA) = \cdeg(\AA_1) \cdot \deg(\AA_2) + \deg(\AA_1) \cdot \cdeg(\AA_2),\]
\[ \ddeg(\AA) = \ddeg(\AA_1) \cdot \deg(\AA_2) + \deg(\AA_1) \cdot \ddeg(\AA_2).\]
\end{enumerate}
\end{Proposition}

\begin{proof}
Let  $(c_i)$ be a minimal reduction of $\C_i$,  for each $i=1,2$. Then $(c)= (c_1)\otimes (c_2)$ is a minimal reduction for $\C$. Thus, the assertion follows from
\[ \C/(c) =  (\C_1/(c_1) \otimes \C_2) \oplus (\C_1 \otimes \C_2/(c_2)), \; \mbox{and} \;
 \C^{**}/\C= (\C_1^{**}/\C_1 \otimes \C_2 ) \oplus  ( \C_1 \otimes \C_2^{**}/\C_2). \qedhere \]
\end{proof}

\begin{Corollary} Let $\AA_i$, $i=1,2$, be the same as in {\rm Proposition~\ref{prod}}. Suppose that $\dim(\AA_{i})=1$ for each $i=1, 2$. Let $r_{i}$ be the type of $\AA_{i}$ for each $i=1, 2$. 
\begin{enumerate}[{\rm (1)}]
\item  Suppose that $\AA_1$ and $\AA_{2}$ are almost Gorenstein. Then 
\[ \cdeg(\AA_1 \otimes_k \AA_2) = (r_1-1)\cdot \deg(\AA_2) + \deg(\AA_1) \cdot (r_2-1).\]
\item If $\AA_{1}$ is almost Gorenstein and $2 \deg(\AA_1) = r_1 + 1$, then $\AA_{1} \otimes \AA_{1}$ is almost Gorenstein.   
\end{enumerate}
\end{Corollary}

\begin{proof} By \cite[Proposition 3.3]{blue1}, we have $\cdeg(\AA_i) = r_{i}- 1$ for each $i=1,2$. Then the assertion follows from Proposition~\ref{prod}.
\end{proof}

\begin{Questions}{\rm
 \begin{enumerate}[(1)]
\item How do we relate $\ddeg(\AA)$ of a graded algebra $\AA$ to $\ddeg(\BB)$ of one of its Veronese subalgebras?
\item If $\SS$ is a finite injective extension of $\RR$, is $\ddeg(\SS) \leq [\SS:\RR]  \cdot \ddeg(\RR)$?
\end{enumerate}
}\end{Questions}

\subsubsection*{Hyperplane sections} \label{hpsection} A  desirable comparison is that   between $\ddeg(\RR)$ and
$\ddeg(\RR/(x))$ for an appropriate regular element $x$. 
 We know that if $\C$ is a canonical module for $\RR$ then
$\C/x \C$ is a canonical module for $\RR/(x)$ with the same number of generators, so type is preserved under
specialization. However $\C/x \C$ may not be isomorphic to an ideal of $\RR/(x)$. Here is a case of good behavior. Suppose
$x$ is regular modulo $\C$. Then from the sequence
\[ 0 \rar \C \rar \RR \rar \RR/\C \rar 0,\]
we get the exact sequence
\[ 0 \rar \C/x\C  \rar \RR/(x) \rar \RR/(\C,x) \rar 0,\] so the canonical module
$\C/x\C$ embeds in $\RR/(x)$. We set $\SS = \RR/(x)$ and $\D = (\C, x)/(x)$ for the image of $\C$ in $\SS$.  We need to compare
$ \ddeg(\RR) = \deg(\RR/\C) - \deg(\RR/\C^{**})$
and $ \ddeg(\SS) = \deg(\SS/\D) - \deg(\SS/\D^{**})$. We can choose $x$ so that $\deg(\RR/\C) = \deg(\SS/\D)$, but need, for instance to show that $\C^{**}$ maps into $\D^{**}$. A model for what we want to  have is
{\cite[Proposition 6.11]{blue1} asserting that if $\C$ is equimultiple then
$ \cdeg(\RR) \leq \cdeg(\RR/(x))$. For bideg$(\RR)$, in accordance with Conjecture~\ref{cdegvsfddeg}
we propose the following.

\begin{Conjecture} {\rm If $\C$ is equimultiple and
$x$ is regular mod $\C$, then $ \ddeg(\RR) \geq \ddeg(\RR/(x))$.
}\end{Conjecture}

\section{Monomial subrings}

\noindent For given nonnegative integers $a_{1}, \ldots, a_{n}$, let  $H = < a_1, \ldots, a_n > = \{c_{1}a_{1} + \cdots + c_{n} a_{n} \mid c_{i} \in \ZZ_{\geq 0}  \}$ be the numerical semigroup generated by $a_{1}, \ldots, a_{n}$. We assume that $\gcd(a_1, \ldots, a_n) = 1$.   For a field $k$, we denote  the subring of $k[t]$
  generated by the monomials $t^{a_i}$ by  $k[H]$.   For reference we shall use \cite[Section 8.7]{VilaBook}.

\medskip

There are several integers playing roles in deriving properties of $k[H]$, with the emphasis on those that lead to the determination of
$\ddeg(k[H])$. There is an integer $s$ such that $t^n\in k[H]$ for all $n\geq s$. The smallest such $s$ is called the {\em conductor} of $H$ or of $k[H]$, which
we denote by $c$ and $t^c k[t]$ is the largest ideal of $k[t]$ contained in $k[H]$.
Then  $c-1$ is called the Frobenius number of $k[H]$ and expresses its multiplicity, $c-1=\deg(k[H])$.

\medskip

For any positive integer $a$ of $H$, the subring $\AA=k[t^a]$ is a Noether normalization of $k[H]$. This permits the passage of many
properties from $\AA$ to $k[H]$, and vice-versa. Then $k[H]$ is a free $\AA$-module and taking into account the natural graded structure
we can write
\[ k[H] = \bigoplus_{j=1}^m \AA t^{\alpha_j}.\]
Note that $s \leq \sum_{j=1}^m  \alpha_j$. 
An interesting question is to determine other invariants of $k[H]$ such as its canonical ideal $\C$ and the reduction number of $\C$ and the canonical degrees
$\cdeg(k[H])$ and $\ddeg(k[H])$.

\subsubsection*{Monomial curves}
Let $k$ be a field and $\RR = k[t^a, t^b, t^c]$ with $a<b< c$ and $\gcd(a,b,c)=1$. Assume $\RR$ is not Gorenstein.  
There exists a surjective  homomorphism $\Gamma: \SS= k[X, Y, Z] \rar \RR$ given by $\Gamma(X)=t^{a}, \Gamma(Y)=t^{b}, \Gamma(Z)=t^{c}$, where $\SS$ is a polynomial ring. 
Let $P= \ker(\Gamma)$.  Then there exists a matrix ${\ds \varphi = \left[ \begin{array}{ll}
X^{a_1} & Y^{b_2} \\
Y^{b_1} & Z^{c_2} \\
Z^{c_1} & X^{a_2}
\end{array}
\right] }$ such that $P=I_{2}(\varphi)$, the ideal generated by $2 \times 2$ minors of $\varphi$. 
That is, 
\[ P = ( f_{1}, f_{2}, f_{3} ) = (X^{a_{1}+a_{2}} - Y^{b_{2}}Z^{c_{1}}, \; Y^{b_{1}+b_{2}} - X^{a_{1}}Z^{c_{2}}, \; Z^{c_{1}+c_{2}} - X^{a_{2}}Y^{b_{1}} )  \]

\medskip

\noindent Let $x, y, z$ be the images of $X, Y, Z$ in $\RR \simeq \SS/P$ respectively.
Suppose $f_{1}, f_{2}$ form a regular sequence and let $L=(f_{1}, f_{2})$.  Then the canonical module of $\RR$ is
\[
 \omega_{\RR}  =   \Ext^{2}_{\SS} ( \RR, \SS)  \simeq   \Ext^{2}_{\SS} ( \SS/P, \SS) \simeq \Hom_{\SS/L}( \SS/P, \SS/L) \simeq (L: P)/ L = (L: f_{3})/L  \simeq (x^{a_{1}}, y^{b_{2}}). \]

\begin{Proposition} \label{preaGor}
Let $\RR= k[t^a, t^b, t^c]$ be a monomial ring. Let ${\ds \varphi = \left[ \begin{array}{ll}
X^{a_1} & Y^{b_2} \\
Y^{b_1} & Z^{c_2} \\
Z^{c_1} & X^{a_2}
\end{array}
\right]}  $ be the Herzog matrix such that $\RR \simeq k[X, Y, Z]/I_{2}(\varphi)$. 
Suppose that $a_1\leq a_2$, $b_{2} \leq b_{1}$, and $c_{1} \leq c_{2}$. Then ${\ds \ddeg(\RR)=  a_{1} b_{2} c_{1} }$. 
\end{Proposition}

\begin{proof} Let   $x, y, z$ be the images of $X, Y, Z$ in $\RR$ respectively. Then we have
\[ x^{a_{1}+a_{2}} = y^{b_{2}}z^{c_{1}}, \quad  y^{b_{1}+b_{2}} = x^{a_{1}}z^{c_{2}}, \quad z^{c_{1}+c_{2}} = x^{a_{2}} y^{b_{1}}. \] 
The ideal  ${\ds \C = (x^{a_{1}}, y^{b_{2}})}$ is the canonical ideal  of $\RR$ and  ${\ds (x^{a_{1}}) : \C = ( x^{a_{1}}, y^{b_{1}}, z^{c_{1}})}$.
Then by Proposition~\ref{bidual},  we have 
\[ \C^{**} = (x^{a_{1}}) : ( (x^{a_{1}}) : \C ) =  (x^{a_{1}}) : ( x^{a_{1}}, y^{b_{1}}, z^{c_{1}}). \] 
Since  $a_1\leq a_2$, $b_{2} \leq b_{1}$, and $c_{1} \leq c_{2}$, we have 
\[ \begin{array}{ll}
{\ds y^{b_{2}} y^{b_{1}} = x^{a_{1}}z^{c_{2}} \in (x^{a_{1}}),} \quad & {\ds  y^{b_{2}}z^{c_{1}} = x^{a_{1}+a_{2}}  \in (x^{a_{1}}), } \vspace{0.1 in} \\
{\ds z^{c_{2}} y^{b_{1}} = z^{c_{1}}y^{b_{2}}z^{c_{2}-c_{1}}y^{b_{1}-b_{2}}   =   x^{a_{1}+ a_{2}} z^{c_{2}-c_{1}}y^{b_{1}-b_{2}}  \in (x^{a_{1}}),} & {\ds  z^{c_{2}}z^{c_{1}} = x^{a_{2}}y^{b_{1}} \in (x^{a_{1}}), } 
\end{array}  \] 
Therefore, ${\ds \C^{**} = (x^{a_{1}}, y^{b_{2}}, z^{c_{2}})}$. Moreover, since $z^{c_{1}+c_{2}} =x^{a_{2}} y^{b_{1}} \in (x^{a_{1}})$, we have   $\C =   (x^{a_{1}}, y^{b_{2}}) = (x^{a_{1}}, y^{b_{2}}, z^{c_{1}+c_{2}}) $. Therefore, 
\[ \begin{array}{rcl}
{\ds   \ddeg(\RR) }  &=& {\ds   \lambda(\RR/\C) - \lambda(\RR/\C^{**}) =   \lambda( \RR/( x^{a_{1}}, y^{b_{2}}, z^{c_{1}+c_{2}}) ) - \lambda( \RR/(x^{a_{1}}, y^{b_{2}}, z^{c_{2}}) ) } \vspace{0.1 in} \\  &= & {\ds a_{1}b_{2}(c_{1}+c_{2}) - a_{1}b_{2}c_{2} =  {a_1} {b_2}   {c_1}. } 
\end{array}   \]
\end{proof}

\begin{Remark}
{\rm For the monomial algebra $\RR=k[t^a, t^b, t^c]$,  the value of $\ddeg (\RR) $ is also
calculated in \cite[Proposition 7.9]{Herzog16}, in accordance with Proposition~\ref{resisddeg1}. 
 According to \cite[Theorem 4.1]{GMP11},  $\cdeg(\RR)$ is either $a_1 b_1 c_1$ or $a_2 b_2 c_2$, which supports  the  Comparison Conjecture~\ref{cdegvsfddeg}.
}\end{Remark}

\section{Rees algebras}\label{Reesalgs}

\noindent Let $\RR$ be a Cohen-Macaulay local ring and $I$ an ideal such that the Rees algebra $S = \RR[I t]$ is Cohen-Macaulay.  We consider
a few classes of such algebras. We denote by $\C$ a canonical ideal of $\RR$. 

\subsubsection*{Rees algebras with expected canonical modules}  This means that the canonical module of $S$ is  $\omega_S = \omega_{\RR}(1,t)^m$, for some $m \geq -1$  (See \cite{HSV87} for details). This will be the case when the associated graded ring  $\gr_I(\RR)$ is Gorenstein (\cite[Theorem 2.4,  Corollary 2.5]{HSV87}). We  assume that  $I$ is an ideal of codimension at least $2$.
We first consider the case $\C = \RR$. Set $\D = (1,t)^m$, pick $a$ such that $a$ is a regular element in $I$ and its initial form $\overline{a}$ is regular on $\gr_I(\RR)$,
and finally  replace $\D$ by $a^m\D\subset S$.

\begin{Proposition}
Let $\RR$ be a Gorenstein local ring, $I$ an ideal of codimension at least two and $S=\RR[It]$ its Rees algebra. If the canonical module of $S$ has the expected
form,  then $S$ is Gorenstein in codimension less than $\codim (I)$; in particular $\ddeg(S) = 0$ and $\cdeg(S)=0$.
\end{Proposition}

\begin{proof}
  It is a calculation in \cite[p.294]{HSV87} that $ S: (1,t)^m = I^m S $. It follows that
\[  (I^m, I^m t) \subset I^m S \cdot (1,t)^m. \]
Since $\codim (I, It) = \codim (I) + 1$, the assertion follows. This implies
 that $\omega_S$ is free in codimension $1$ and therefore it is reflexive by a standard argument. 
 Finally by Theorem~\ref{genddeg1}, $\ddeg( S) = 0$,
 as   $S$
 is Gorenstein in codimension $1$.
  Therefore $\cdeg(S)=0$ by \cite[Corollary 2.4]{blue1}.
\end{proof}

\subsubsection*{Rees  algebras of minimal multiplicity} Let $(\RR, \m)$ be a  Cohen-Macaulay local ring, $I$ an $\m$-primary ideal,  and $J$ a minimal reduction of $I$ with $I^2 = JI$. 
From the exact sequence
 \[  0\rar I\RR[Jt] \rar \RR[J] \rar \RR[Jt]/I\RR[Jt] = \gr_J(\RR) \otimes_{\RR} \RR/I 
 \rar 0,\] 
  with the middle and right terms being Cohen-Macaulay, we have that $I\RR[Jt] = I \RR[It]$ is  Cohen-Macaulay. 
With $\gr_I(\RR)$ Cohen-Macaulay and $I$ an ideal of reduction number at most $1$,  
we have  if $\dim \RR > 1$ then $\RR[I t]$ is Cohen-Macaulay (see \cite[Section 3.2.1]{intclos}). 
A source of these ideals arises from irreducible ideals in Gorenstein rings such as  $I= J:\m$, where  $J$ is a parameter ideal.

\medskip

Let $S_0 = \RR[Jt]$ and $S = \RR + I t S_0$.
If $\RR$ is Gorenstein, then the canonical module of $S_{0}$ is  $\omega_{S_0} = (1,t)^{d-2}S_0$ and by the change of rings formula, we have
\[ \omega_{S} = \Hom_{S_0}( S, (1,t)^{d-2}S_0)= (1,t)^{d-2}S_0: S.\]

If $d=2$, then the canonical module $\omega_S$ is the conductor of $S$ relative to $S_0$. In the case
$ \m S_0$ is a prime ideal of $S_0$, so the conductor could not be larger as it would have grade at least two and then
$S = S_0$:  $\omega_S = \m S_0$.
If $d > 2$, then  $\m (1,t)^{d-2}S_0$ will
 work.

\medskip

Note that $S_0$ is Gorenstein in codimension $1$. If $P$ is such prime and   $P \cap \RR=\q\neq \m$,
then $(S_0)_\q = S_\q$, so $S_P=(S_0)_P$ is Gorenstein in codimension $1$. Thus we may assume
  $P\cap \RR = \m $, so that $P = \m S_0$.
   For $S_P$ to be Gorenstein  would mean $\Hom(S, S_0) \simeq S$ at $P$, that is ${\ds (S_0:  S)_P }$ is a principal 
   ideal of $S_P$ (see next Example). From \cite[Theorem 3.3.7]{BH},
 $\D = \Hom(S, S_0) = \m S_0=\m S$.   Thus $\deg(S/\D) = 2$ and since $\D^{**} \neq \D$, $\deg(S/\D^{**}) = 1$.
It follows that $\ddeg(S) = 1$.

\begin{Example}{\rm  Let $\RR = k[x,y]$,  and $I = (x^3, x^2 y^2, y^3)$. Then for the reduction $Q = (x^3,y^3)$, we have $I^2 = QI$. The Rees
algebra $S = \RR[It]= k[x,y,u,v,w]/L$, where 
$L=(x^2 u-xv, y^2w-y v, v^2-xyuw)$ is given by the $2\times 2$ minors of
${\ds 
\varphi = \left[ \begin{array}{ll}
v & xw \\
yu & v \\
x & y
\end{array}
\right],
}$
whose content is $(x,y,v)$. It follows that $S$ is not Gorenstein in codimension $1$ as it would require a content of codimension at least four.  
Indeed, setting $\AA = k[x,y, u, v, w]$, a projective resolution of $S$ over $\AA$ is defined by
the mapping $\varphi:  \AA^2 \rar \AA^3$ which dualizing gives
\[ \C=\Ext^2_{\AA}(\AA/L, \AA) = \Ext^1_{\AA}(L, \AA)= \coker(\varphi^{*}).\]
It follows that $\C$ is minimally generated by two elements
at the localizations of $\AA$ that contain $(x,y,v)$.
}\end{Example}

\subsubsection*{Rees algebras of ideals with the expected defining relations} Let $I$ be an ideal of the Cohen-Macaulay local ring $(\RR,\m)$ (or a polynomial ring $\RR= k[t_1, \ldots, t_d]$ over the field
$k$)
with a presentation
\[ \RR^m \stackrel{\varphi}{\lar} \RR^n \lar I =(b_1, \ldots, b_n)\rar 0.\]  Assume that $\codim (I) \geq 1$ and that the entries of $\varphi$ lie in $\m$. 
Denote by $\LL$ its ideal of relations
\begin{eqnarray*}
 0 \rar \LL \lar \SS = \RR[\TT_1, \ldots, \TT_n] \stackrel{\psi}{\lar}
\RR[It] \rar 0,  \quad \TT_i \mapsto b_it .
\end{eqnarray*} Then $\LL$ is a graded ideal of $\SS$ and its component in degree $1$ is generated by the $m$ linear forms
\[ \ff= [\ff_1, \ldots, \ff_m] = [\TT_1, \ldots , \TT_n] \cdot \varphi.\]
\noindent  Let
$\mathbf{f}= \{\ff_1, \ldots, \ff_s\} $ be a set of polynomials in
 $\LL\subset \SS=\RR[\TT_1, \ldots , \TT_n]$ and let $\mathbf{a}= \{a_1, \ldots,
a_s\}\subset \RR $. If  $\ff_i \in (\mathbf{a}) \SS$ for all $i$, we can
write
\[ \mathbf{f}= [\ff_1, \ldots, \ff_s] = [a_1, \ldots, a_q] \cdot\AA =
\mathbf{a}\cdot \AA,\] where $\AA $ is an $s\times q$ matrix with entries
in $\SS$. We call $(\aa)$ a $\RR$-content of $\mathbf{f}$.
 Since  $\aa\not \subset \LL$, then the $s \times s$ minors of $\AA$ lie in $\LL$.
   By an abuse of terminology\index{Sylvester form}, we refer
   to such determinants
    as
{\em
Sylvester forms}, or the {\em Jacobian
duals}\index{Jacobian dual},  of $\mathbf{f}$ relative to $\mathbf{a}$.
If
 $\aa = I_1(\varphi)$, we write $\AA = \BB(\varphi)$, and call it the Jacobian dual of  $\varphi$. Note that if $\varphi$ is a matrix with linear entries
 in the variables $x_1, \ldots, x_d$, then
 $\BB(\varphi)$ is a matrix with linear entries in the variables $\TT_1, \ldots, \TT_n$.

\begin{Definition}\label{expectrel} {\rm Let $I = (b_1, \ldots, b_n)$ be an ideal with a presentation as above and let = $\aa  (a_1, \ldots, a_s) = I_1(\varphi)$. The Rees algebra
$\RR[It]$ has the {\em expected relations} if
\[ \LL =  ( \TT \cdot \varphi, I_s(\BB(\varphi))).\]
}\end{Definition}

There will be numerous restrictions to ensure that $\RR[It]$ is a Cohen-Macaulay ring and that it is amenable to the determination of its canonical degrees. We consider
some special cases grounded on \cite[Theorem 1.3]{MU} and \cite[Theorem 2.7]{UlrichU96}.

\begin{Theorem}\label{rees1} Let $\RR = k[x_1, \ldots,   x_d]$ be a polynomial ring over an infinite field and
let $I$ be a perfect $\RR$-ideal of grade $2$ with a linear presentation matrix $\varphi$. 
 Assume that
 $\nu(I) > d$
 and that $I$ satisfies  $G_d$ $($meaning $\nu(I_{\pp})\leq \dim \RR_{\pp}$ on the punctured spectrum$)$.
  Then $\ell(I) = d$,  $r(I) = \ell(I)-1$, $\RR[It]$ is Cohen-Macaulay, and $\LL = (\TT\cdot \varphi, I_d(\BB(\varphi))).$
 \end{Theorem}

 The canonical ideal of these rings is described in \cite[Theorem 2.7]{UlrichU96}.

 \begin{Proposition} 
 Let $I$ be an ideal of codimension two satisfying the assumptions of {\rm Theorem~\ref{rees1}}.     
 Then $\ddeg(\RR[It]) \neq 0$
and similarly $\cdeg(\RR[It])\neq 0$. 
 \end{Proposition}
 
 \begin{proof}
Let $J$ be a minimal reduction of $I$ and set $K = J:I$. Then $ \C =K t  \RR[It]$ is a canonical ideal of $\RR[It]$ by \cite[Theorem 2.7]{UlrichU96}.
Let $a$ be a regular element of $K$. Then
 \[ (a): \C = \sum L_i t^i  \subset \RR[It], \quad \mbox{and} \quad bt^i \in L_i \;\;  \mbox{if and only if} \;\; b  KI^j \subset a I^{i+j}.\]
Thus for $b\in L_0$, $b\cdot K \subset (a)$ so $b= r\cdot a$ since $\codim (K) \geq 2$. Hence $L_0 = (a)$.
 For $b t\in L_1$, $b K \subset aI$, $b= ra$ with $r \in I:K$. As $rKI \subset I^2$, we have $L_1 = a(I:K)t$.
 For $i \geq 2$, $b\in L_i$, we have $b = rat^i$ with $r K \subset I^i$. Hence $L_i =a (I^i:K)t^i$.  In
 general we have $L_i = a(I^{i}: K)t^i$. Therefore, we have
 \[ \begin{array}{rcl}
 {\ds (a): \C} &=& {\ds  a(\RR+ (I:K) t + (I^2:K)t^2 +  \cdots ),} \vspace{0.1 in} \\
{\ds \C^{**} } &=& {\ds   (a) : ((a): \C) =  \sum_j L_jt^j = \RR[It] :   ( \RR + (I:K) t + (I^2:K)t^2 +  \cdots  ) }.
\end{array} \]
For $b\in L_0$ and $i\geq 1$,  $b(I^i:K) \subset I^i$ and thus $b\in \bigcap_i I^i : (I^i:K)=L_0$.  In general it is clear that $KI^i\subset L_{i}$. Note that
$L_0 = \RR: (\RR:K) = \RR$. It follows that $\C \neq \C^{**}$ and hence $\ddeg(\RR[It]) \neq 0$.
Recall that the vanishing of either of the functions $\cdeg$ or $\ddeg$ holds if and only if the ring is Gorenstein in codimension $1$. Therefore $\cdeg(\RR[It]) \neq 0$.
 \end{proof}

\section{Calculations in $\AA=\m:\m$}
\noindent
Let $(\RR, \m)$ be a Cohen Macaulay local ring of dimension $1$ and let $\AA=\m:\m$. 
The driving questions of this section
 are what are the type $r(\AA)$ of $\AA$, $\cdeg(\AA)$ and $\ddeg(\AA)$ in relation to
 the invariants of $\RR$.  We begin by collecting elementary data on $\AA$.

 \begin{Theorem}\label{NGens} Let $(\RR, \m)$ be a Cohen Macaulay local ring of dimension $1$ with a canonical ideal $\C$.   Let $Q$ the total ring of fractions of $\RR$ and $\AA  = \m:_ Q \m$.  Suppose that $\RR$ is not a DVR. 
 \begin{enumerate}[{\rm (1)}]
 \item We have  ${\ds \AA = x^{-1} ((x):_{\RR} \m )}$ for any regular element $x\in \m$.
 \item The minimal number of generators of $\AA$ is  $r+1$, where $r$ is the type of $\RR$.
 \item The ideal $\m \C$ is the canonical ideal of $\AA$. 
 \end{enumerate}
 \end{Theorem}

\begin{proof} (1) Since $\RR$ is not a DVR, we have $\m (\RR:_{Q} \m) \neq \RR$. Hence
\[ \AA =  \m:_{Q} \m \subset \RR:_{Q} \m \subset \m:_{Q} \m. \]
Therefore, $\AA = \RR :_{Q} \m$ and by Proposition~\ref{bidual},  for any regular element $x \in \m$, we have
\[ \AA =  \RR:_{Q} \m= x^{-1} ((x):_{\RR}\m).\]
(2) Let $x \in \m \setminus \m^{2}$ be a regular element. Then the type of $\RR$ is
\[ r = \dim (\mbox{Soc}(\RR/\RR x ) ) = \lambda ( ( (x) :_{\RR} \m )/(\RR x) ). \]
Note that $\m x  \subset \m( (x):_{\RR} \m) \subset \RR x $ and $\lambda(\RR x / \m x )= 1$. 
If $ \m( (x):_{\RR} \m) = \RR x $, then $x \in  \m( (x):_{\RR} \m)  \subset \m^{2}$, which is a contradiction. Therefore $\m x  = \m( (x):_{\RR} \m) $ and the minimal number of generators of $\AA$ is 
\[ \begin{array}{rcl}
{\ds  \nu(\AA) } &=& {\ds  \nu(((x):_{\RR} \m ) ) = \lambda (  ((x):_{\RR} \m )/\m ((x):_{\RR} \m )   ) = \lambda (  ((x):_{\RR} \m )/\m x   ) } \vspace{0.1 in} \\
&=& {\ds   \lambda ( ( (x) :_{\RR} \m )/(\RR x) ) +  \lambda(\RR x / \m x ) = r+1. }
\end{array}   \]
(3) Let $k = \RR/\m$.  Since $\m$ is an ideal of both $\RR$ and $\AA$,  we have $\RR:\AA = \m$. Thus $\AA/\RR \simeq k^n$ for some $n$.
The exact sequence ${\ds 0 \rar \RR \rar \AA \rar k^n \rar 0}$ yields
\[ 0\rar k  \rar \AA/\m\AA \rar k^n \rar 0.\]
By (2), we have  $r+1=\nu(\AA) = n+1$ so that $n=r$. Now by applying $\Hom(\cdot, \C)$ to the exact sequence ${\ds 0 \rar \RR \rar \AA \rar k^r \rar 0}$, 
we get
\[ 0\rar \Hom(\AA, \C)  \rar \C \rar \Ext^{1}(k^{r}, \C) \simeq k^{r} \rar 0. \]
Let $\D= (\C :_{Q} \AA) \simeq \Hom(\AA, \C)$, which is the canonical module of $\AA$ (\cite[Theorem 3.3.7]{BH}).  Then $\C/\D \simeq k^r$ and $\m \C \subset \D $. 
Since $r$ is the type of $\RR$, we have $\C/ \m \C \simeq k^{r}$. Thus, $\D=\m C$.
\end{proof}

\begin{Remark}\label{semilocal}{\rm Let $\RR$ and $\AA$ be the same as in Theorem~\ref{NGens}. 
A relevant point is to know when $\AA$ is a local ring. Let us  briefly consider some cases.

\begin{enumerate}[(1)]

\item  Let $L$ be an ideal of the local ring $\RR$
and suppose
 $L = I \oplus J$, $L = I + J$, $I\cap J = 0$, is a non-trivial decomposition.
  Then $I \cdot (J: I) = 0$ and $J\cdot (J:I) = 0$ and thus if
 $\grade(I,J) = 1$ (maximum possible by the Abhyankar-Hartshorne Lemma), then $I:J=0$. It follows
 that
 \[ \Hom(L,L) = \Hom(I,I) \times \Hom(J,J),\] and therefore $\Hom(L,L)$ is not a local ring.

  \item   Suppose $\RR$ is complete, or at least Henselian. If $\AA$ is not a local ring, by
  the Krull-Schmidt Theorem,  $\AA$ admits a non-trivial decomposition of $\RR$-algebras
  ${\ds  \AA = \BB \times \CC}$. 
 Since $\m = \m \AA$, we have a decomposition $\m = \m\BB \oplus \m \CC$. If we preclude  such decompositions,  then
 $\AA$ is a local ring. In general, if the completion of $\RR$ is a domain, then a similar argument would
 apply. We refer to  \cite[Chapter 7]{Eisenbudbook} for a treatment of such rings.
\end{enumerate}
}\end{Remark}

Our main result is the following.

\begin{Theorem} \label{TCdeg}
Let $\RR$ and $\AA$ be  the same as  in {\rm Theorem~\ref{NGens}}. 
Suppose $(\AA, \M)$  is a local ring. Let  $e=[\AA/ \M:\RR/\m]$ and $r$ the type of $\RR$. Then  
\[ \cdeg(\AA) = e^{-1} ( \cdeg(\RR) + \rme_0(\m) - 2r ).\] 
\end{Theorem}

\begin{proof} By Theorem~\ref{NGens}, the ideal $\m \C$ is the canonical ideal of $\AA$.
Let $(x)$ and $(c)$ be minimal reductions of $\m$ and $\C$ respectively. Then $(xc)$ is a minimal reduction of $\m \C$. 
Thus, we have
\[ \cdeg(\AA) = \lambda_{\AA}(\AA/xc\AA) - \lambda_{\AA}(\AA/\m \C) = e^{-1} (   \lambda_{\RR}(\AA/xc\AA) - \lambda_{\RR}(\AA/\m \C)  ). \]
Moreover, we have the following.
\[ \begin{array}{rcl}
{\ds \lambda_{\RR}(\AA/xc\AA)  } &=& {\ds \lambda_{\RR}(\AA/x\AA)+ \lambda_{\RR}(x\AA/xc\AA) = \lambda_{\RR}(\RR/ x\RR) + \lambda_{\RR}(\RR/c\RR), } \vspace{0.1 in} \\
{\ds  \lambda_{\RR}(\AA/\m \C) } &=& {\ds \lambda_{\RR}(\AA/\RR) + \lambda_{\RR}(\RR/\C) + \lambda_{\RR}(\C/\m \C) = 2r + \lambda_{\RR}(\RR/\C).  } 
\end{array} \]
Since $\rme_{0}(\m)= \lambda_{\RR}(\RR/x \RR)$ and $\cdeg(\RR)= \lambda_{\RR}( \RR/c \RR) - \lambda_{\RR}(\RR/ \C)$, we obtain 
\[ \cdeg(\AA) = e^{-1} ( \cdeg(\RR) + \rme_0(\m) - 2r ). \qedhere \]
\end{proof}

 \begin{Corollary} Let $\RR$ and $\AA$ be the same as in {\rm Theorem~\ref{TCdeg}}.  
Then  $\AA$ is a Gorenstein ring if and only if
 ${\ds \cdeg(\RR) + \rme_0(\m) - 2r=0}$. 
   In particular 
   $\m$ and $\C \AA$ are principal ideals of $\AA$. 
\end{Corollary}

\begin{proof} Suppose the canonical ideal of $\AA$ is  $\m \C = y \AA$. Then $\m (\C \AA) = y \AA$ as the product of  $\AA$-ideals. Thus, both $\m$ and $\C\AA$ are invertible ideals of the local ring $\AA$.
\end{proof} 

\begin{Remark}{\rm  Let $\RR$ and $\AA$ be the same as  in Theorem~\ref{NGens}.  Suppose that $\AA$ is semilocal with maximal ideals $\M_1, \ldots, \M_s$. 
We can still obtain a formula for $\cdeg(\AA)$ as a summation of the $\cdeg(\AA_{\M_i})$, that is
\[ \cdeg(\AA) =  \sum_{i=1}^{s} \cdeg (\AA_{\M_i}) = (\cdeg(\RR) + \rme_0(\m) - 2r) \sum_{i=1}^{s} e_i^{-1},\]
where ${\ds e_{i}= [\AA/ \M_{i} : \RR/\m]}$ for each $i$. 
}\end{Remark}

 \begin{Example}{\rm
Let $(\RR,\m)$ be a Stanley-Reisner ring of dimension $1$. If $\RR=k[x_1, \ldots, x_n]/L$, $L$ is generated by all the quadratic monomials
 $x_ix_j$, $i\neq j$. Note that $\m = (x_1)\oplus \cdots \oplus (x_n)$ and since for $i\neq j$ the annihilator
 of $\Hom((x_i), (x_j))$ has grade positive, by Remark~\ref{semilocal} we have
 \[ \AA =\Hom(\m,\m) = \Hom((x_1), (x_1))\times \cdots \times \Hom((x_n),(x_n))= k[x_1] \times \cdots \times k[x_n].\]
In order to compute $\cdeg(\RR)$, first we note that $\rme_0(\m)=n$. 
Since the Hilbert series of $\RR$ is  ${\ds {{1 + (n-1)t}\over {1-t}}}$, we have $r = n-1$. Therefore,
\[ \cdeg(\RR) = 2r -\rme_0(\m) = 2(n-1)-n = n-2=r-1. \]
Thus $\RR$ is almost Gorenstein (\cite[Theorem 5.2]{blue1}). This also follows from
 \cite[Theorem 5.1]{GMP11}.
}\end{Example}

Let us consider other invariants of $\AA = \m : \m$.

\begin{Remark}{\rm
It would be interesting  to calculate $\ddeg(\AA)$  under the same conditions as in  in Theorem~\ref{TCdeg}. That is, 
${\ds \ddeg(\AA)= \lambda_{\AA}(\AA/\D) - \lambda_{\AA}(\AA/\D^{**})}$, 
where $\D= \m \C$. Note that  ${\ds  \lambda_{\AA}(\AA/\D) = e^{-1} (2 r+ \lambda(\RR/ \C))}$.
}\end{Remark}

\begin{Example}{\rm Consider the monomial ring  $\RR = \mathbb{Q}[t^5, t^7, t^9]$. 
We have a presentation $\RR  \simeq \mathbb{Q}[X, Y, Z]/P $, with $P = (Y^2-XZ, X^5-YZ^2, Z^3-X^4 Y)$.  Let $x, y, z$ be the images of $X, Y, Z$ in $\RR$ respectively.  
Let us examine some properties of  $\AA = \m:\m$, where $\m=(x, y, z)$.
 
\begin{enumerate}[(1)]
\item As we discussed in Example~\ref{monoex}, $\RR$ has the canonical ideal $\C = (x, y)$ and $\cdeg(\RR)=2$. 
 Moreover we have $\rme_0(\m) = 5$,  $r=2$ and $e=1$. Therefore $\cdeg(\AA)= 2+5-4 = 3$.

\item  Let $\D$ be the canonical ideal of $\AA$. To calculate $\D^{**}$ we change $\C$ to $x \C$. We then get $\lambda(\AA/\D) = 11$ and $\lambda(\AA/\D^{**})= 9$. Therefore
 $\ddeg(\AA) = \lambda(\AA/\D)-\lambda(\AA/\D^{**}) = 2$.
\end{enumerate}
}\end{Example}

\section*{Acknowledgement}

We would like to thank the referee for his/her thoughtful and meticulous feedback and comments.

\end{document}